\documentclass[11pt]{amsart}
\usepackage{amsthm,wrapfig,amsmath,amssymb,amsfonts,setspace,verbatim,graphicx,mathtools,mathrsfs,commath,float,mdframed,frame,xcolor,qtree,enumitem, ulem}
\usepackage[hidelinks]{hyperref}
\usepackage[T1]{fontenc}

\usepackage[margin=1.35in]{geometry}

\newtheorem{ut}{Theorem}
\numberwithin{ut}{section}

\newtheorem{uc}[ut]{Corollary}

\newtheorem{uq}{Question}

\theoremstyle{remark}
\newtheorem{ur}{Remark}
\numberwithin{ur}{section}

\theoremstyle{definition}

\begin{document}

\title[A note on the topology of escaping endpoints]{A note on the topology of escaping endpoints}

\address{Department of Mathematics, Auburn University at Montgomery, Montgomery 
AL 36117, United States of America}
\email{dsl0003@auburn.edu,dlipham@aum.edu}
\author{David S. Lipham}

\begin{abstract}
 We  study topological properties of  the escaping endpoints and fast escaping endpoints of  the Julia set of complex exponential $\exp(z)+a$  when  $a\in (-\infty,-1)$.  We show neither space is homeomorphic to the whole set of endpoints. This follows from a general result stating that for every transcendental entire function $f$,  the escaping  Julia set $I(f)\cap J(f)$ is first category. 
\end{abstract}

\maketitle

\section{Introduction}

For each $a\in (-\infty,-1)$ define $f_a:\mathbb C\to \mathbb C$ by  $f_a(z)=e^z+a$.  The Julia set $J(f_a)$ is known to be a \textit{Cantor bouquet} consisting of an uncountable union of pairwise disjoint rays, each joining a finite endpoint to the point at infinity \cite[p.50]{dev}; see Figure 1. Let $E(f_a)$ denote the set of finite endpoints of these rays.  Mayer  proved $E(f_a)\cup \{\infty\}$ is connected, even though $E(f_a)$ is totally disconnected \cite{may}.  The one-point compactification $J(f_a)\cup \{\infty\}$ is a Lelek fan \cite{aa}, so $E(f_a)$ is actually homeomorphic to the ``irrational Hilbert space'' $\mathfrak E_{\mathrm{c}}:=\{x\in \ell^2:x_i\notin \mathbb Q\text{ for each }i<\omega\}$ \cite{31}, which is \textit{almost zero-dimensional} \cite{ov2,erd}.  This means  $E(f_a)$ has a basis of open sets whose closures are intersections of clopen sets. We note that by \cite[Theorem 3.1]{lip} and \cite{aa,bul, cha},  an almost zero-dimensional space $X$ has a one-point connectification if and only if $X$ is homeomorphic to a dense set $X'\subseteq E(f_a)$ with the property that  $X'\cup \{\infty\}$ is connected.

Alhabib and Rempe-Gillen recently discovered that $\dot E(f_a)\cup \{\infty\}$ is connected, where $\dot E(f_a)$  is  the set of \textit{escaping endpoints} of $J(f_a)$ \cite[Theorem 1.4]{rem}.   The even smaller set of \textit{fast escaping endpoints} $\ddot E(f_a)$ also has the property that its union with $\{\infty\}$ is connected \cite[Remark  p.68]{rem}.  More can be said about the topologies of  $\dot E(f_a)\cup \{\infty\}$ and $\ddot E(f_a)\cup \{\infty\}$ based on \cite{lip}. For example $\dot E(f_a)\cup \{\infty\}\setminus K$ is connected for every $\sigma$-compact set $K\subseteq \dot E(f_a)$ \cite[Theorem 4.7]{lip}. The primary goal of this paper is to investigate whether $\dot E(f_a)$ and $\ddot E(f_a)$ are topologically equivalent to  $\mathfrak E_{\mathrm{c}}$ or the ``rational Hilbert space'' $\mathfrak E:=\{x\in \ell^2:x_i\in\mathbb Q\text{ for each }i<\omega\}$. We show both sets are first category in themselves, implying   neither space is homeomorphic to  $\mathfrak E_{\mathrm{c}}$. We also show $\ddot E(f_a)\not\simeq \mathfrak E$.  It is presently  unknown whether $\dot E(f_a)$ is homeomorphic to $\mathfrak E$.

\section{Preliminaries}

Let $f$ be an entire function. 
\begin{itemize}\renewcommand{\labelitemi}{\scalebox{.5}{{$\blacksquare$}}}

\item A set $X\subseteq \mathbb C$ is:
\subitem\textit{backward-invariant} under  $f$ provided $f^{-1}[X]\subseteq X$;
\subitem\textit{forward-invariant} under  $f$ provided $f[X]\subseteq X$; and
\subitem\textit{completely invariant} under $f$ if $f^{-1}[X]\cup f[X]\subseteq X$.  

\item The \textit{backward orbit} of a point $z\in \mathbb C$ is the union of pre-images $$O^-(z)=\bigcup\{f^{-n}\{z\}:n<\omega\}.$$  The \textit{forward orbit} of $z$ is the set $O^+(z)=\{f^n(z):n<\omega\}$.

\item A point $z\in \mathbb C$ is \textit{exceptional} if $O^-(z)$ is finite.  There is at most one exceptional point  \cite[p.6]{ber2}. 

\item $I(f)=\{z\in \mathbb C:f^n(z)\to\infty\}$ is called the  \textit{escaping set} for  $f$.

\item Define the maximum modulus function
$M(r) := M(r, f) = \max \{|f(z)|:|z|=r\}$ for $r \geq 0$. 
Choose $R > 0$ sufficiently large such that $M^n(R) \to +\infty$ as $n \to\infty$ and  let 
$A_R(f) = \{z \in \mathbb C : |f^n(z)| \geq M^n(R)\text{ for all }n \geq0\}$. The \textit{fast escaping set} for $f$ is defined to be the increasing union of closed sets 
$$A(f) =\bigcup _{n\geq 0}   f^{-n}[A_R(f)].$$
It can be shown that the definition of $A(f)$ is independent of the choice of $R$ when $f$ is transcendental \cite[Theorem 2.2]{rip}.  

\item Note that $J(f)$, $I(f)$, and  $A(f)$   are completely invariant under $f$.

\end{itemize}

Recall that for each $a\in (-\infty,-1)$,  the endpoint set of $J(f_a)$ is denoted  $E(f_a)$. We let  \begin{align*}\dot E(f_a)&=I(f_a)\cap E(f_a)\text{; and}\\
\ddot E(f_a)&=A(f_a)\cap E(f_a)
\end{align*} denote  the  \textit{escaping endpoints} and \textit{fast escaping endpoints} of $J(f_a)$, respectively. 
\begin{figure}[h]
	\centering
	\includegraphics[scale=0.22]{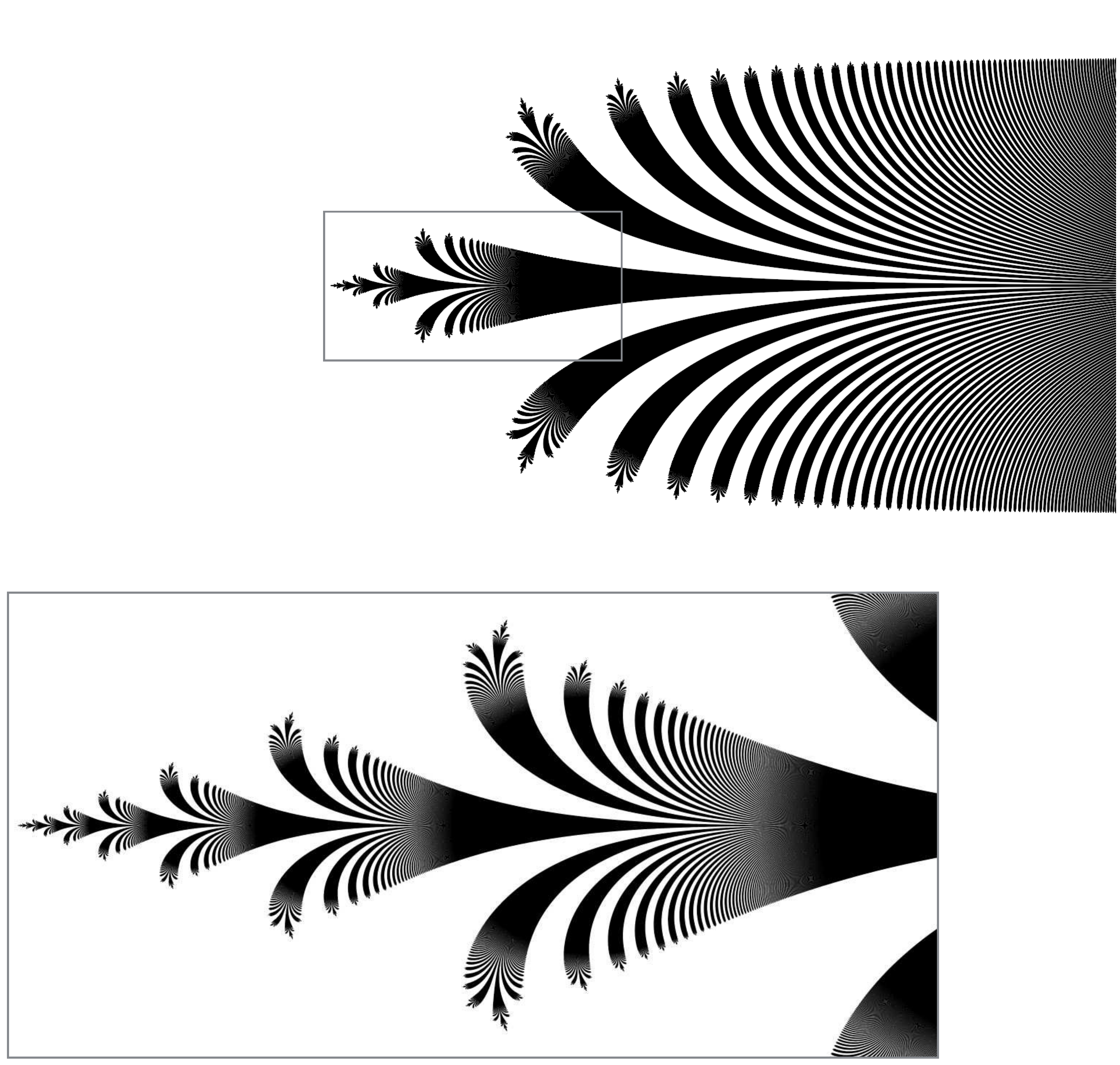}
	\caption{Partial images of $J(f_{-2})$}
\end{figure}

\section{Results for transcendental entire functions}

In this section we assume $f$ is a transcendental entire function, so that $I(f)\cap J(f)\neq\varnothing$ \cite[Theorem 2]{ere}.  We will make repeated uses of \cite[Lemma 4]{ber2}, which states that $\overline{O^-(z)}=J(f)$ for each non-exceptional point $z\in J(f)$. This is a  simple consequence of Montel's theorem. A topological space $X$ is \textit{first category} if $X$  can be written as the union of countably many (closed) nowhere dense subsets.  

\begin{ut}Every completely invariant subset of $I(f)\cap J(f)$ is first category.\end{ut}

\begin{proof} Let $X\subseteq I(f)\cap J(f)$ be completely invariant under $f$. Let $R=|z_0|+1$ for some $z_0\in J(f)$. For each $n<\omega$ let $X_n=\{z\in X:|f^k(z)|\geq R\text{ for all }k\geq n\}.$  Since $X\subseteq I(f)$, we have $X=\bigcup \{X_n:n<\omega\}.$  It remains to show each $X_n$ is nowhere dense in $X$. To that end, fix $n<\omega$.   Let $U$ be any open subset of $J(f)$ such that $U\cap X\neq\varnothing$. We will  show $U\cap X\not\subseteq X_n$.  

For any point $z\in I(f)$ the forward orbit $O^+(z)$ is infinite. Since $X$ is forward-invariant, it contains $O^+(z)$ when $z\in X$. We assume $X$ is non-empty, so   $X$ is infinite.  There is at most one exceptional point by Picard's theorem,  so there is a non-exceptional point $z_1\in X$.  By \cite[Lemma 4]{ber2}, $O^-(z_1)$ contains a dense subset of  $\{z\in J(f): |z|<R\}$, which is a perfect set \cite[Theorem 3]{ber2}.  So there  is a non-exceptional point $z_2\in  O^-(z_1)$ with $|z_2|<R$.  The set of repelling periodic points is a dense subset of $J(f)$  \cite[Theorem 4]{ber2}.   Since $I(f)$ contains no periodic point,  we have $\overline{J(f)\setminus I(f)}=J(f)$. For each $k<\omega$ we also note that $f^{-k}\{z_2\}$ is closed and  $f^{-k}\{z_2\}\subseteq I(f)$. So each pre-image $f^{-k}\{z_2\}$ is nowhere dense in $J(f)$. Therefore  $V:=U\setminus\bigcup \{f^{-k}\{z_2\}:0\leq k<n\}$ is a non-empty open subset of $J(f)$. By \cite[Lemma 4]{ber2}    there exists $k<\omega$ such that $f^{-k}\{z_2\}\cap V\neq\varnothing$. Then $k\geq n$ and $f^{-k}\{z_2\}\cap U\neq\varnothing$. Let $z_3\in f^{-k}\{z_2\}\cap U$. Then $|f^{k}(z_3)|=|z_2|<R$, so $z_3\notin X_n$. Since $X$ is backward-invariant,   $z_3\in (U\cap X)\setminus X_n$ as desired.  Clearly $X_n$ is a relatively closed subset of $X$. We conclude that $(U\cap X)\setminus X_n$ is a non-empty $X$-open subset of $U\cap X$ missing $X_n$.   Recall $U$ was an arbitrary open subset of $J(f)$ intersecting $X$, so this proves $X_n$ is nowhere dense in $X$. \end{proof}

\begin{uc}$I(f)\cap J(f)$ is first category.\end{uc}

\begin{proof}Theorem 3.1 applies since $I(f)\cap J(f)$ is completely invariant under $f$.  \end{proof}

\begin{uc}$J(f)\setminus I(f)$ is not first category.\end{uc}

\begin{proof}$J(f)$ is a closed subset of $\mathbb C$, and is therefore not the union of two first category sets. Since $I(f)\cap J(f)$ is first category (Corollary 3.2), $J(f)\setminus I(f)$ is not.  \end{proof}

\section{Applications  to complex exponentials $f_a$}

\begin{ut}$I(f_a)$, $A(f_a)$,  $\dot E(f_a)$ and $\ddot E(f_a)$  are first category.\end{ut}

\begin{proof}These are completely invariant subsets of $I(f_a)$.  And $I(f_a)\subseteq J(f_a)$; this actually holds for all $a\in \mathbb C$ \cite[Section 2]{el}.  So Theorem 3.1 applies to each set.\end{proof}

\begin{ur}Complete invariance of $\dot E(f_a)$ was also applied in \cite[p.68]{rem} to generalize the main result in  \cite{may}. \end{ur}

\begin{ut}Neither $\dot E(f_a)$ nor $\ddot E(f_a)$ is homeomorphic to $E(f_a)$.\end{ut}

\begin{proof}$E(f_a)$ is completely metrizable (recall $E(f_a)\simeq \mathfrak E_{\mathrm{c}}$, which is a $G_\delta$-subset of $\ell^2$), so by the Baire Category Theorem $E(f_a)$ is not first category. The result now follows from Theorem  4.1.  \end{proof}

\begin{ut}$\ddot E(f_a)\not\simeq \mathfrak E$.\end{ut}

\begin{proof}$\ddot E(f_a)$ is an absolute $G_{\delta\sigma}$-space because $A(f_a)$ and $E(f_a)$ are $F_\sigma$ and $G_\delta$ subsets of $\mathbb C$, respectively. On the other hand, $\mathfrak E$ is not  absolute $G_{\delta\sigma}$ because it has a  closed subspace  homeomorphic to $\mathbb Q ^\omega$; see \cite[p.23]{erd}.\end{proof}

\begin{uq}Is $\dot E(f_a)$ homeomorphic to $\mathfrak E$?\end{uq}

\begin{uq}Is $\ddot E(f_a)$ homeomorphic to $\mathbb Q\times \mathfrak E_{\mathrm{c}}$?\end{uq}

\subsection*{\normalfont\textsc{Acknowledgements}}Lasse Rempe-Gillen supplied the graphics for Figure 1.  Philip Rippon indicated a strengthening of Corollary 3.3.  Namely, $J(f)\setminus I(f)$ contains a dense $G_\delta$-subset  of $J(f_a)$ by \cite[Lemma 1]{bak}.


\begin{thebibliography}{HD}

\bibitem{aa}J. M. Aarts, L. G. Oversteegen, The geometry of Julia sets, Trans. Amer. Math. Soc.
338 (1993), no. 2, 897--918.

\bibitem{rem}N. Alhabib, L. Rempe-Gillen, Escaping Endpoints Explode. Comput. Methods Funct.
Theory 17, 1 (2017), 65--100.


\bibitem{ber2}W. Bergweiler, Iteration of meromorphic functions, Bulletin of the American Mathematical Society, Volume 29 Number 2, October 1993, Pages 151--188.

\bibitem{bul}W. D. Bula and L. G. Oversteegen, A characterization of smooth Cantor bouquets, Proc. Amer. Math. Soc. 108 (1990), 529--534. 

\bibitem{cha}W. J. Charatonik, The Lelek fan is unique, Houston J. Math. 15 (1989), 27--34.

\bibitem{dev}R. L. Devaney and M. Krych, Dynamics of $\exp(z)$, Ergodic Theory Dynamical Systems 4 (1984), 35--52.

\bibitem{lip}J. J. Dijkstra and D. S. Lipham, On cohesive almost zero-dimensional spaces, preprint.

\bibitem{erd}J. J. Dijkstra, J. van Mill, Erd\H{o}s space and homeomorphism groups of manifolds, Mem. Amer. Math. Soc. 208 (2010), no. 979.

\bibitem{bak}I. N. Baker and P. Domínguez, Residual Julia Sets, Journal of Analysis, Vol. 8, 2000, pp. 121--137.


\bibitem{el}A. E. Eremenko and M. Y. Lyubich, Dynamical properties of some classes of entire functions, Ann. Inst. Fourier (Grenoble) 42 (1992), no. 4, 989--1020.

\bibitem{ere}A. E. Eremenko, On the iteration of entire functions,  Banach Center Publications 23.1 (1989): 339--345.

\bibitem{31}K. Kawamura, L. G. Oversteegen, and E. D. Tymchatyn, On homogeneous totally disconnected $1$-dimensional spaces, Fund. Math. 150 (1996), 97--112.


\bibitem{may} J. C. Mayer,  An explosion point for the set of endpoints of the Julia set of $\lambda\exp(z)$. Ergod. Theory
Dyn. Syst. 10(1),  (1990) 177--183.

\bibitem{ov2}L. G. Oversteegen and E. D.  Tymchatyn. On the dimension of certain totally disconnected spaces. Proceedings of the American Mathematical Society, 122(3), (1994) 885--891.

\bibitem{rip}P. J. Rippon and G. M. Stallard, Fast escaping points of entire functions, Proc. London Math. Soc., 105 (2012), 787--820.

\end{thebibliography}
\end{document}